\documentclass[10pt,reqno]{amsart}

\usepackage{amsmath}
\usepackage{amsthm}
\usepackage{amssymb}
\usepackage{amsfonts}
\usepackage{amsxtra}
\usepackage{fullpage}
\usepackage{amscd}
\usepackage{color} 
\usepackage{bm} 
\usepackage{mathrsfs}
\usepackage{subfigure}
\usepackage{tikz}
\usepackage{tikz-cd}
\usetikzlibrary{cd}
\usepackage{enumerate}
\usepackage[all]{xy}
\usepackage[mathscr]{eucal}
\usepackage{verbatim}

\let\nc\newcommand
\let\renc\renewcommand

\theoremstyle{plain}
\newtheorem{thm}{Theorem}
\newtheorem{prop}[thm]{Proposition}
\newtheorem{cor}[thm]{Corollary}
\newtheorem{lem}[thm]{Lemma}

\theoremstyle{definition}
\newtheorem{defn}[thm]{Definition}

\numberwithin{thm}{section}

\nc{\bdm}{\begin{displaymath}}
\nc{\edm}{\end{displaymath}}
\nc{\bthm}{\begin{thm}}
\nc{\ethm}{\end{thm}}
\nc{\blem}{\begin{lem}}
\nc{\elem}{\end{lem}}
\nc{\bcor}{\begin{cor}}
\nc{\ecor}{\end{cor}}
\nc{\bprop}{\begin{prop}}
\nc{\eprop}{\end{prop}}
\nc{\bdef}{\begin{defn}}
\nc{\eddef}{\end{defn}}

\makeatletter
\renewcommand{\subsection}{\@startsection{subsection}{2}{0pt}{-3ex
plus -1ex minus -0.2ex}{-2mm plus -0pt minus
-2pt}{\normalfont\bfseries}} \makeatother

\numberwithin{equation}{section}

\DeclareMathOperator{\gr}{\mathrm{gr}}

\DeclareMathOperator{\rad}{\mathrm{rad}}

\DeclareMathOperator{\Tr}{\mathrm{Tr}}

\newcommand{\beq}{\begin{equation}\label}
\newcommand{\eeq}{\end{equation}}

\newcommand{\iso}{{\;\stackrel{_\sim}{\to}\;}}

\DeclareMathOperator{\Hom}{\mathrm{Hom}}

\nc{\Z}{\mathbb{Z}}

\newcommand{\Q}{\mathbb{Q}}

\newcommand{\C}{\mathbb{C}}
\newcommand{\h}{\mathfrak{h}}

\nc{\rank}{\textrm{rank} \,}
\nc{\ds}{\dots}

\let\mc\mathcal
\let\mf\mathfrak

\nc{\mbf}{\mathbf}
\nc{\Res}{\mathsf{Res} \, }
\nc{\Ind}{\mathsf{Ind} \, }
\nc{\cont}{\textrm{cont}}

\nc{\msf}{\mathsf}

\nc{\minusone}{-1}
\nc{\minustwo}{-2}
\nc{\Mod}{\mathrm{Mod} \,}
\nc{\ms}{\mathscr}
\nc{\Frac}{\mathrm{Frac} \,}
\nc{\ra}{\rightarrow}
\nc{\hra}{\hookrightarrow}
\nc{\lab}{\label}
\renc{\O}{\mc{O}}
\nc{\Tan}{\mc{T}}
\nc{\ul}{\underline}
\nc{\s}{\mathfrak{S}}
\nc{\g}{\mf{g}}
\nc{\pa}{\partial}
\nc{\tit}{\textit}
\nc{\Maxspec}{\mathrm{Maxspec} \, }
\nc{\gldim}{\mathrm{gl.dim}}
\nc{\rkm}{\mathrm{rk} \, (\mf{m})}
\nc{\sm}{\mathrm{sm}}
\nc{\PD}{\mathbb{PD}}
\nc{\hilb}{\textrm{Hilb}}
\nc{\<}{\langle}
\renc{\>}{\rangle}
\nc{\T}{\mathbb{T}}
\nc{\X}{\mathbb{X}}
\nc{\F}{\mathbb{F}}
\nc{\id}{\msf{id}}
\nc{\A}{\mathbb{A}}
\nc{\Grat}{\mc{Grat}}
\nc{\Squo}[1]{\A^{(#1)}}
\nc{\twist}{\mathrm{twist}}
\nc{\Cd}{\mc{C}}
\nc{\Span}{\mathrm{Span}}
\nc{\Grass}{\mathrm{Gr}}
\nc{\Supp}{\mathrm{Supp}}
\nc{\Irr}{\mathrm{Irr}}
\renc{\o}{\otimes}
\renc{\gr}{\mathsf{gr}}
\nc{\fin}{\mathrm{fin}}
\nc{\aff}{\mathrm{aff}}
\nc{\algD}{\mf{D}}
\nc{\hr}{\mf{h}_{\textrm{reg}}}
\nc{\D}{\mathscr{D}}
\nc{\PIdeg}{\mathrm{PI-degree}}
\nc{\ch}{\mathrm{ch}}
\nc{\ev}{\mathsf{ev}}
\nc{\Stab}{\mathrm{Stab}}
\nc{\Der}{\mathrm{Der}}
\nc{\rightsim}{\stackrel{\sim}{\longrightarrow}}
\nc{\HZ}{H_{\mbf{h},\Z}(\Z_m)}
\nc{\sing}{\mathrm{sing}}
\nc{\dd}{\mathscr{D}}
\nc{\bc}{\mathbf{c}}
\nc{\vc}{\underline{\mathbf{c}}}
\nc{\ba}{\mathbf{a}}
\nc{\reg}{\mathrm{reg}}
\nc{\Amp}{\mathrm{Amp}}
\nc{\Nef}{\mathrm{Nef}}
\nc{\SL}{\mathrm{SL}}
\nc{\Sp}{\mathrm{Sp}}
\nc{\Sym}{\mathrm{Sym}}
\nc{\Mov}{\mathrm{Mov}}
\nc{\Pic}{\mathrm{Pic}}
\nc{\Cs}{\C^{\times}}
\nc{\Nak}[3]{\mf{M}_{{#1}} ({#2},{#3}) }
\nc{\Naka}[2]{\mf{M}({#1},{#2}) }
\nc{\Mtheta}[1]{\mc{M}_{#1}}

\nc{\bw}{\mathbf{w}}

\nc{\bn}{\mathbf{n}}
\nc{\CB}{\mathrm{CB}}
\nc{\GVect}{\Lambda}
\nc{\pZ}{\overline{Z}}
\nc{\Tang}{\mc{T}}
\nc{\K}{\mathbb{K}}

\newcommand{\mr}{\mathrm}
\newcommand{\Soc}{\mathrm{Soc}}

\nc{\red}[1]{\textcolor{red}{#1}}

\begin{document}

\title{The Rank-One property for free Frobenius extensions}

\author[G. Bellamy]{Gwyn Bellamy}
\address{School of Mathematics and Statistics, University of Glasgow, University Place,
Glasgow, G12 8QQ.}
\email{gwyn.bellamy@glasgow.ac.uk}

\author[U. Thiel]{Ulrich Thiel}
\address{Department of Mathematics, University of Kaiserslautern, Postfach 3049, 67653 Kaiserslautern, Germany}
\email{thiel@mathematik.uni-kl.de}

\begin{abstract}
A conjecture by the second author, proven by Bonnaf\'e-Rouquier, says that the multiplicity matrix for baby Verma modules over the restricted rational Cherednik algebra has rank one over $\mathbb{Q}$ when restricted to each block of the algebra. 

In this paper, we show that if $H$ is a prime algebra that is a free Frobenius extension over a regular central subalgebra $R$, and the centre of $H$ is normal Gorenstein, then each central quotient $A$ of $H$ by a maximal ideal $\mathfrak{m}$ of $R$ satisfies the rank-one property with respect to the Cartan matrix of $A$. Examples where the result is applicable include graded Hecke algebras, extended affine Hecke algebras, quantized enveloping algebras at roots of unity, non-commutative crepant resolutions of Gorenstein domains and 3 and 4 dimensional PI Sklyanin algebras.

In particular, since the multiplicity matrix for restricted rational Cherednik algebras has the rank-one property if and only if its Cartan matrix does, our result provides a different proof of the original conjecture.\\ 

\noindent R\'esum\'e. Une conjecture du deuxi\`eme auteur, qui a \'et\'e prouv\'ee par Bonnaf\'e-Rouquier, dit que la matrice de multiplicit\'e des b\'eb\'e modules de Verma de l'alg\`ebre rationnelle restreinte de Cherednik est de rang un sur $\mathbb{Q}$ lorsqu'elle est restreinte \`a chaque bloc de l'alg\`ebre.
	
Dans cet article nous montrons que si $H$ est une alg\`ebre premi\`ere qui est une extension libre de Frobenius sur une sous-alg\`ebre centrale r\'eguli\`ere $R$, et si le centre de $H$ est Gorenstein normal, alors chaque	quotient central $A$ de $H$ par un id\'eal maximal $\mathfrak{m}$ de $R$ satisfait la propri\'et\'e de rang un par rapport \`a la matrice de Cartan de $A$. Les exemples o\`u le r\'esultat est applicable incluent les alg\`ebres de Hecke gradu\'ees, les alg\`ebres de Hecke affines \'etendues, les alg\`ebres enveloppantes quantifi\'ees aux racines de l'unit\'e, les r\'esolutions	cr\'epantes non commutatives des domaines de Gorenstein et les alg\`ebres PI Sklyanin \`a dimension $3$ et $4$.
	
	En particulier, puisque la matrice de multiplicit\'e pour les alg\`ebres
	de Cherednik rationnelles restreintes a la propri\'et\'e de rang un si et
	seulement si sa matrice de Cartan l'a aussi, notre
	r\'esultat fournit une preuve diff\'erente de la conjecture originale.
\end{abstract}

\maketitle

\section{Introduction}

\subsection{}\label{sec:intro1.1} It was conjectured by the second author \cite[Question.~2(ii)]{CHAMP} that baby Verma modules over the restricted rational Cherednik algebra satisfy a remarkable rank-one property. Namely, if $\Delta(\lambda)$, resp. $L(\lambda)$, denotes the baby Verma module, resp. irreducible module, associated to $\lambda \in \Irr W$, then the \textit{multiplicity matrix} $M$, given by 
$$
M_{\lambda,\mu} := [\Delta(\lambda):L(\mu)],
$$
has rank one over $\Q$ when restricted to each block of the algebra. This conjecture was confirmed by Bonnaf\'e-Rouquier \cite[Proposition 14.4.2]{BonnafeRouquier}, whose ingenious proof makes use of a splitting extension of the centre of the (un-restricted) rational Cherednik algebra, together with properties the corresponding decomposition map to the restricted rational Cherednik algebra. Let $C$ denote the \textit{Cartan matrix} of the algebra. The baby Verma modules for the restricted rational Cherednik algebra satisfy BGG reciprocity, $C = M^T M$, from which it is easily seen (Section~\ref{sec:exaples}) that the rank-one property can equivalently be stated as saying that the Cartan matrix of each block of the restricted rational Cherednik algebra has rank one. This reformulation makes no explicit mention of baby Verma modules. 

\subsection{} In this article we give a completely different proof of the rank-one property. Our proof applies to a broader class of algebras, including for instance quantum groups at roots of unity, and makes use of the fact that these algebras are free Frobenius $R$-algebras, for an appropriate regular central subalgebra $R$. In particular, our result applies to all but two of the families of examples considered in \cite{BGS}. 

In the introduction, we assume for simplicity that $\K$ is an algebraically closed field of characteristic zero; in the body of the paper we prove a more general result that does not require this assumption. We start with $H$ a (unital) affine $\K$-algebra and central subalgebra $R \subset H$ such that $H$ is a free Frobenius extension of $R$; see Section~\ref{sec:Nakayamaext}. 

We assume that $R$ is regular, $H$ is prime and the centre $Z := Z(H)$ of $H$ is Gorenstein and integrally closed. Let $\mf{m}$ be a maximal ideal of $R$ and $A = H  / \mf{m} H$. The $\K$-algebra $A$ is finite dimensional and split. Let $K_0(A)$ be the Grothendieck group of finitely generated projective $A$-modules and $G_0(A)$ the Grothendieck group of all finitely generated $A$-modules. The Cartan map is the canonical map $C \colon K_0(A) \to G_0(A)$. If $\Lambda$ is a (finite) set parametrizing isomorphism classes of simple $A$-modules, with representatives $L(\lambda)$ for each $\lambda \in \Lambda$, then we can think of $C$ as a $|\Lambda| \times |\Lambda|$ matrix whose entries are the multiplicities $[P(\lambda): L(\mu)]$, where $P(\lambda)$ is the projective cover of $L(\lambda)$. If $A = A_1 \oplus \cdots \oplus A_k$ is the block decomposition of $A$ then $C$ admits a corresponding decomposition $C = C_1 \oplus \cdots \oplus C_k$, where $C_i$ is the Cartan matrix of $A_i$.

We say that the algebra $A$ has the \textit{rank-one property} if each matrix $C_i$ has rank one over $\Q$. Though this appears at first to be a very strong property, our main result is:

\begin{thm}\label{thm:rankonemain}
 $A$ has the rank-one property. 
\end{thm}

The proof of Theorem~\ref{thm:rankonemain} centres around the Higman map $\tau' \colon H \to H$, which descends to the Higman map $\tau \colon A \to A$. The key result is Theorem~3(a) of Lorenz-Fitzgerald Tokoly \cite{LorenzFitzgerald} (see also \cite{LorenzFDHopf}), which says that the rank of the Cartan matrix $C$ equals the dimension of $\mr{Im}\, \tau$ as a $\K$-vector space. We show that this dimension equals the number of blocks of $A$. We do this by identifying $\mr{Im}\, \tau$ with the socle of the image $Z_{\alpha}'$ in $A$ of the Nakyama centre $Z_{\alpha}$ of $H$. 

\subsection{Examples} In the motivating case, $H = H_c(W)$ is the rational Cherednik algebra at $t = 0$ associated to the complex reflection group $(\h,W)$. This $\C$-algebra contains $R = \C[\h]^W \o \C[\h^*]^W$ as a central subalgebra and $Z$ is an integrally closed Gorenstein ring \cite[Theorem~3.3, Lemma~3.5]{EG}. It has been shown in \cite{BGS} that $H$ is a (symmetric) free Frobenius extension of $R$. If $\mf{m} \subset \C[\h]^W \o \C[\h^*]^W$ is the augmentation ideal, then the central quotient $A = H / \mf{m} H = \overline{H}_c(W)$ is the restricted rational Cherednik algebra. Since this algebra satisfies BGG reciprocity, the rank-one property with respect to the Cartan matrix $C$ is equivalent to the rank-one property with respect to the multiplicity matrix $M$. 

Theorem~\ref{thm:rankonemain} applies to many other examples commonly studied in representation theory. These include graded Hecke algebras, extended affine Hecke algebras, affine nil-Hecke algebras, quantized enveloping algebras at roots of unity, non-commutative crepant resolutions of Gorenstein domains and 3 and 4 dimensional PI Sklyanin algebras; see section~\ref{sec:otherexamples}.

\section{The proof}

We begin again, dropping any assumptions on the field $\K$, unless specifically stated. Theorem~\ref{thm:rankonemain} will follow from the more general Theorem~\ref{thm:rankonemainbody} below. 

\subsection{The Nakayama centre}\label{sec:Nakayamaext}

We start with $H$ a (unital) ring and a central subring $R \subset H$ such that $H$ is a free Frobenius extension of $R$. By definition, this means that $H$ is a finite free $R$-module and there is an isomorphism 
\begin{equation}\label{eq:Hbiiso}
\phi \colon {}_1 H_{\alpha^{-1}} \stackrel{\sim}{\longrightarrow} \Hom_R(H,R)
\end{equation}
of $H$-bimodules, for an $R$-linear ring automorphism $\alpha$ of $H$. Here ${}_1 H_{\alpha^{-1}} = H$ as left $H$-modules, but with right action given by $h \cdot a = h \alpha^{-1}(a)$. The automorphism $\alpha$ is unique up to inner automorphisms and one can check that $\alpha |_{Z} = \mr{Id}_Z$, where $Z := Z(H)$. Let 
\[
Z_{\alpha}(H) = \{ h \in H \, | \, ha = \alpha(a) h \, \textrm{for all} \, a \in H \};
\]
this is the \textit{Nakayama centre} of $H$. We note that under the isomorphism $\phi$ of \eqref{eq:Hbiiso}, 
\begin{equation}\label{eq:Zalphaiso}
    Z_{\alpha}(H) \iso \Hom_R(H/ [H,H], R).
\end{equation}
Since $\alpha |_{Z} = \mr{Id}_Z$, we have an identification of $Z$-modules $Z_{\alpha}(H) \cong \Hom_R(H/ [H,H], R)$. 

Assume that $H$ is a prime ring. Then it is a prime PI ring since it is a finite $R$-module. Let $d$ denote the PI degree of $H$. Let $T(H)$ denote the trace ring of $H$ (see \cite[Section~13.9]{MR}) and $\Tr \colon H \to T(H)$ the reduced trace. 
The following result is \cite[Proposition~2.3]{BraunSymmetric}. 

\begin{lem}\label{lem:BrauncocentreNakayam}
	Assume that $T(H) = H$ and $d$ is invertible in $H$. Then the map
	$$
	\Hom_R(H/[H,H],R) \stackrel{\sim}{\longrightarrow} \Hom_R(Z,R), \quad f \mapsto f |_Z,
	$$
	is an isomorphism of $Z$-modules, with inverse $g \mapsto g \circ (d^{-1} \mr{Tr})$. 
\end{lem} 

Here the map $d^{-1} \mr{Tr} \colon H \to Z$ is a projection, realizing $Z$ as a direct summand of $H$ as $Z$-modules. 

\begin{lem}\label{lem:Nakayamcentresplit}
	Assume that $T(H) = H$ and $d$ is invertible in $H$. Then $Z_{\alpha}(H)$ is a direct summand of $H$ as $Z$-modules. 
\end{lem}

\begin{proof}
	Since $Z_{\alpha}(H)$ corresponds to $\Hom_R(H/[H,H],R)$ under the isomorphism $\phi$, it suffices to show that $\Hom_R(H/[H,H],R)$ is a summand of $\Hom_R(H,R)$ as $Z$-modules. Consider the series of maps 
	$$
	Z \hookrightarrow H \to H / [H,H] \stackrel{d^{-1} \mr{Tr}}{\longrightarrow} Z 
	$$
	whose composition $Z \to Z$ is the identity. They dualise to 
$$
\Hom_R(Z,R) \stackrel{(d^{-1} \mr{Tr})^*}{\longrightarrow} \Hom_R(H / [H,H],R) \to \Hom_R(H,R) \twoheadrightarrow \Hom_R(Z,R),
$$ 
where the composition $\Hom_R(Z,R) \to \Hom_R(Z,R)$ must again be the identity. Here we have used the fact that $Z$ is a summand of $H$ to conclude that the last arrow is surjective. Since $(d^{-1} \mr{Tr})^*$ is an isomorphism by Lemma~\ref{lem:BrauncocentreNakayam}, we deduce that $\Hom_R(H / [H,H],R) \to \Hom_R(H,R)$ is a split injection.  
\end{proof}

Let $\mf{m}$ be a maximal ideal of $R$ and $A = H / \mf{m} H$. Let $Z'$, respectively $Z_{\alpha}'$, denote the image of $Z$, respectively of $Z_{\alpha}(H)$, in $A$. Lemma~\ref{lem:Nakayamcentresplit} implies that the map $Z_{\alpha}(H) \to A$ realises $Z_{\alpha}(H) / \mf{m} Z_{\alpha}(H)$ as a direct summand of $A$.

\begin{lem}\label{lem:mainrankone}
	Let $\mf{m} \subset R$ be a maximal ideal, contained in the regular locus. Assume that:
	\begin{enumerate}
		\item The integer $d$ is invertible as an element of $H$. 
		\item $Z$ is Gorenstein and integrally closed.   
	\end{enumerate}
	Then $Z / \mf{m} Z$ is a Frobenius algebra, the morphism $Z / \mf{m} Z \to Z'$ is an isomorphism and $Z_{\alpha}(H) / \mf{m} Z_{\alpha}(H) \cong Z_{\alpha}'$ is a free rank-one $Z'$-module.
\end{lem}

\begin{proof}
	First we note that the fact that $Z$ integrally closed implies that the trace ring $T(H)$ of $H$ equals $H$; see \cite[Proposition~13.9.5]{MR}. Then, as noted previously, the fact that the integer $d$ is invertible as an element of $H$ means that $Z$ is a direct summand of $H$ as a $Z$-module. It follows that the canonical map $Z / \mf{m} Z \to Z'$ is an isomorphism.  
	
	Since $H$ is a free $R$-module, it follows that $Z$ is a projective $R$-module. Let $R_{\mf{m}}$ denote the localization of $R$ at $\mf{m}$ and $Z_{\mf{m}} = Z \o_R R_{\mf{m}}$. Then $Z_{\mf{m}}$ is a free $R_{\mf{m}}$-module. Therefore, we need to show that $Z / \mf{m} Z = Z_{\mf{m}} / {\mf{m}}Z_{\mf{m}}$ is Frobenius. Since $R_{\mf{m}}$ is assumed to be a regular local ring, $\mf{m} R_{\mf{m}}$ is generated by a regular sequence $f_1, \ds, f_k$. This sequence is also regular in $Z_{\mf{m}}$. It follows from Proposition~3.1.19(b) of \cite{CohMac} that $Z_{\mf{m}} / {\mf{m}}Z_{\mf{m}}$ is zero-dimensional Gorenstein. This means it is a Frobenius algebra. 
	
	The hypothesis of Lemma~\ref{lem:BrauncocentreNakayam} hold and hence  $\Hom_R(H/[H,H],R) \cong \Hom_R(Z,R)$ as $Z$-modules. Now, $\Hom_R(Z,R)_{\mf{m}} \cong \Hom_{R_{\mf{m}}}(Z_{\mf{m}},R_{\mf{m}})$ as $Z_{\mf{m}}$-modules. Since $H$ is prime, $Z$ is a domain and hence equi-dimensional. Also, $R$ is a regular ring and thus Gorenstein. Under these hypothesis, \cite[Proposition~2.6]{BraunSymmetric} says that $Z_{\mf{m}}$ being Gorenstein is equivalent to $\Hom_{R_{\mf{m}}}(Z_{\mf{m}},R_{\mf{m}}) \cong Z_{\mf{m}}$ as $Z_{\mf{m}}$-modules. Since $\Hom_R(H/[H,H],R)$ is isomorphic to $Z_{\alpha}(H)$ by \eqref{eq:Zalphaiso}, we deduce that $Z_{\alpha}(H)_{\mf{m}} \cong Z_{\mf{m}}$. This implies that $Z_{\alpha}(H)/ \mf{m} Z_{\alpha}(H)$ is isomorphic to $Z / \mf{m} Z$. 
\end{proof}

\subsection{The main theorem}

Let $\K$ denote the residue field of $R$ at $\mf{m}$. Then $A$ is a finite-dimensional $\K$-algebra. We assume that $A$ is split over $\K$. Let $p \ge 0$ denote the characteristic of $\K$.  

The centre of $A$ is denoted $Z(A)$. Clearly $Z' \subset Z(A)$, but in general the inclusion is strict. As in the introduction, $A = A_1 \oplus \cdots \oplus A_k$ is the block decomposition of $A$ and $C = C_1 \oplus \cdots \oplus C_k$ the corresponding decomposition of the Cartan matrix. We set $C_{\K} \colon K_0(A) \o_{\Z} \K \to G_0(A) \o_{\Z} \K$ with decomposition $C_{\K} = C_{\K,1} \oplus \cdots \oplus C_{\K,k}$. We call the rank of $C_{\K}$ over $\K$ the \textit{$p$-rank} of $C$ (abuse of language).

\begin{thm}\label{thm:rankonemainbody}
	If the $p$-rank of every $C_{\K,i}$ is non-zero then the $p$-rank of each $C_{\K,i}$ is exactly one. 
\end{thm}

We note that if $\K$ has characteristic zero then it is automatic that the rank ($=$ $p$-rank) of each block of $C$ is at least one. Therefore, Theorem~\ref{thm:rankonemain} is a direct consequence of Theorem~\ref{thm:rankonemainbody}. 

\subsection{Proof of Theorem~\ref{thm:rankonemainbody}}\label{sec:proofmainrankone}

Recall that we have assumed the following hold. 
\begin{enumerate}
	\item $H$ is a free Frobenius $R$-algebra. \label{list:H1}
	\item $R$ is a regular ring. \label{list:H2}
	\item $Z = Z(H)$ is Gorenstein and integrally closed. \label{list:H3} 
  	\item $A := H/\mf{m} H$ is split over $\K := R / \mf{m}$, for $\mf{m} \lhd R$ maximal. \label{list:H4}
	\item The $p$-rank of every block $C_i$ of $C$ is non-zero. \label{list:H5}
	\item $H$ is a prime ring whose PI degree is invertible in $H$.  \label{list:H6}
\end{enumerate}

The isomorphsim \eqref{eq:Hbiiso} defines an $R$-linear pairing on $H$ by $(g,h) := \phi(1)(gh)$. We fix a pair $\{ g_j : j \in I\}$ and $\{h_j : j \in I\}$ of dual $R$-bases for $H$, meaning that $(g_i,h_j) = \delta_{i,j}$. This allows us to define the \textit{Higman map} 
\[
\tau \colon H \to H, \quad \tau(x) = \sum_{j \in I} g_j x h_j.
\]
By \cite[Lemma~2.13]{LorenzBook}, the image of $\tau$ is a $Z$-submodule of $Z_{\alpha}(H)$. 

The $R$-linear automorphism $\alpha$ descends to a $\K$-linear automorphism $\alpha'$ of $A$ and the isomorphism \eqref{eq:Hbiiso} induces an isomorphism ${}_{1} A_{(\alpha')^{-1}} \iso \Hom_{\K}(A,\K)$. Thus, $A$ is also Frobenius with $\{ g_j : j \in I\}$ and $\{h_j : j \in I\}$ giving dual bases. Again, there is a Higman map $\tau' : A \to A$, $\tau'(x) = \sum_{j \in I} g_j x h_j$, with image $\mr{Im}\, \tau'$ contained in $Z_{\alpha'}(A)$. In the case where $A$ is symmetric, this image is called the Higman ideal or projective centre of $A$. 

By construction, there is a commutative diagram
\begin{equation}\label{eq:tautauprime}
\begin{tikzcd}
	H  \ar[r] \ar[d,"\tau"'] & A \ar[d,"{\tau'}"] \\
	Z_{\alpha}(H) \ar[r,"\iota"] & Z_{\alpha'}(A). 
\end{tikzcd}	
\end{equation}
Note that the image of $\iota$ is $Z_{\alpha}'$, which is identified with $Z_{\alpha}(H) / \mf{m} Z_{\alpha}(H)$ by Lemma~\ref{lem:mainrankone}. Recall that the socle $\Soc_H M$ of an $H$-module $M$ is the sum of all simple submodules. 

\begin{lem}\label{lem:imtauprimesoc}
	$\mr{Im}\, \tau' \subset \Soc_{Z'} Z_{\alpha}'$. 
\end{lem}


\begin{proof}
	First we note that $(\Soc_A A) \cap Z_{\alpha}' \subset \Soc_{Z'} Z_{\alpha}'$. Next, by diagram~\eqref{eq:tautauprime}, the image of $\tau'$ is contained in $Z_{\alpha}'$ since the image of $\tau$ is contained in $Z_{\alpha}(H)$. On the other hand, it is explained in Section 2.4.2 of \cite{LorenzFitzgerald} (see also \cite[Proposition~2.20]{LorenzBook}) that the image of $\tau'$ is contained in $\Soc_A A$. The lemma follows.  
\end{proof}

Let $Z' = B_1 \oplus \cdots \oplus B_l$ denote the blocks of $Z'$. Since $Z'$ is Gorenstein, Lemma~\ref{lem:mainrankone} implies that the socles of the indecomposable summands $B_i Z_{\alpha}'$ of $Z_{\alpha}'$ are simple. 

We note that $Z(A)$ is split over $\K$ because we have assumed in \eqref{list:H4} that $A$ is split over $\K$; see \cite[Exercise 7.5]{LamFirstcourse}. Since $Z' \subset Z(A)$, the algebra $Z'$ is also split over $\K$ \cite{LamFirstcourse}. Thus, every simple $Z'$-module is one dimensional over $\K$ and it follows from Lemma~\ref{lem:imtauprimesoc} that $\dim_{\K} \mr{Im}\, \tau' \le \ell(\Soc_{Z'} Z_{\alpha}')$, the length of $\Soc_{Z'} Z_{\alpha}'$. By Lemma~\ref{lem:mainrankone}, the latter equals $\ell(\Soc_{Z'} Z')$. Since we have assumed by \eqref{list:H3} that $Z'$ is Frobenius, the number of blocks of $Z'$ equals the number of simple modules in the socle of $Z'$. Thus, we have shown that   
$$
\dim_{\K} \mr{Im}\, \tau' \le |\mr{Bl}(Z')|. 
$$ 
Using once again \eqref{list:H4} that $A$ is split over $\K$, the key result Theorem 3(a) of \cite{LorenzFitzgerald} says that $\dim_{\K} \mr{Im}\, \tau'$ equals the $p$-rank of the Cartan map $C_{\K}$. Assumption \eqref{list:H5} implies that the $p$-rank of $C$ (the rank of the map $C_{\K}$) is at least as big as $|\mr{Bl}(A)|$. M\"uller's Theorem \cite[Proposition~2.7]{Ramifications} implies that $|\mr{Bl}(Z')| = |\mr{Bl}(A)|$. Thus, $\dim_{\K} \mr{Im}\, \tau \ge |\mr{Bl}(A)|$ and hence 
$$
\dim_{\K} \mr{Im}\, \tau' \le |\mr{Bl}(Z')| = |\mr{Bl}(A)| \le \dim_{\K} \mr{Im}\, \tau'. 
$$
We deduce that $|\mr{Bl}(A)| = \dim_{\K} \mr{Im}\, \tau'$ equals the $p$-rank of $A$. Since the $p$-rank of each block is at least one, we have proven Theorem~\ref{thm:rankonemainbody}.   

We note that a consequence of the proof is that $\mr{Im} \, \tau' = \Soc_{Z'} Z_{\alpha}'$. 

\subsection{The Casimir map}

We also have the Casimir map $q \colon A \to A$ given by $q(a) = \sum_j h_j a g_j$. By \cite[Lemma~2.13]{LorenzBook}, the image of $q$ is contained in $Z(A)$. It is a consequence of \cite[Proposition 2.20]{LorenzBook} that both $\tau'$ and $q$ vanish on $\rad A$ and have image in $\Soc_{A} A$. 

If $H$ is the free Frobenius extension $\C[x] \rtimes \s_2$ of $\C[x^2]$ and $\mf{m} = (x^2)$ then $A = \C[x] / (x^2) \rtimes \s_2$ is precisely the example considered in Exercise~2.2.3 of \cite{LorenzBook}. This is also a special case of the graded Hecke algebras considered in \cite{BGS}. In this example, $\mr{Im}\, \tau'$ is one-dimensional whilst $q$ is the zero map. Thus, the rank of $\tau'$ does not equal the rank of $q$ in general for a Frobenius algebra. This example also shows that the image of $\tau'$ need not be central.

\section{Examples}\label{sec:exaples}

\subsection{Triangular decompositions} In this section, we use freely the notation from \cite{BelThielHighest}. We begin by justifying the claim made in Section~\ref{sec:intro1.1} that if $H$ is a free Frobenius extension such that $A$ is a graded algebra with triangular decomposition then the multiplicity matrix $M$ has the rank-one property if and only if the Cartan matrix $C$ does.

To be precise, we assume that $H$ is $\Z$-graded, $R$ a graded (central) subalgebra and $\mf{m}$ a homogeneous maximal ideal of $R$ such that $A = H / \mf{m} H$ admits a triangular decomposition $A = A^- \o T \o A^+$ as in \cite[Definition~3.1]{BelThielHighest}. Let $B^{\pm}$ be the subalgebras of $A$ generated by $A^{\pm}$ and $T$.

\begin{lem}\label{lem:MCrankone}
Assume that $T$ is semi-simple and $B^- \cong (B^+)^{\circledast}$ as graded $T$-bimodules. Then $A$ satisfies the rank-one property with respect to the multiplicity matrix $M$ if and only if it does so with respect to the Cartan matrix $C$. 
\end{lem} 

Indeed, in this situation, $C = M^T M$ by BGG-reciprocity \cite[Theorem~1.3]{BelThielHighest}. Since $M$ is integer (and hence real) valued, the rank of $C$ equals the rank of $M$ and it follows that each block of $M$ has rank one if and only if each block of $C$ does.

Thus, for graded algebras with a triangular decomposition, the rank-one property can be encoded as 
\begin{equation}\label{eq:rankonemult}
[\Delta(\lambda) : L(\rho) ] \dim_{\K} \Delta(\mu) = [\Delta(\mu) : L(\rho) ] \dim_{\K} \Delta(\lambda)	
\end{equation}
for all $\lambda,\rho, \mu \in \Irr T$. Indeed, the fact that each block of $M$ has rank one implies that there exist rational numbers $a_{\lambda},b_{\lambda}$ such that $[\Delta(\lambda) : L(\rho)] = a_{\lambda} b_{\rho}$. If $d := \sum_{\mu} b_{\mu} \dim_{\K} L(\mu)$, then $\dim_{\K} \Delta(\lambda) = a_{\lambda} d$. Hence, both sides of \eqref{eq:rankonemult} equal $a_{\lambda} a_{\mu} b_{\rho} d$.

In particular, Lemma~\ref{lem:MCrankone} applies to $H = H_c(W)$, the rational Cherednik algebra and $\mf{m} \lhd R := \C[\h]^W \o \C[\h^*]^W$ the augmentation ideal. The quotient $A = H / \mf{m} H = \overline{H}_{c}(W)$ is the restricted rational Cherednik algebra.

\subsection{Other examples}\label{sec:otherexamples} In this section, we assume all algebras are defined over an algebraically closed field of characteristic zero. In addition to rational Cherednik algebras (I), most examples in \cite{BGS} satisfy the hypothesis of Theorem~\ref{thm:rankonemainbody}. Numbered as in \cite{BGS}, they are:
\begin{enumerate}
    \item[(II)] Graded Hecke algebras with $R = Z$ regular.
    \item[(III)] The extended affine Hecke algebra with $R = Z$.
    \item[(IV)] Affine nil-Hecke algebra with $R = Z$. 
    \item[(V)] Quantized enveloping algebra at an $\ell$th root of unity with $R = Z_0$ the $\ell$-centre. The $\ell$-centre is the localization of a polynomial ring \cite[III, Theorem~6.2(2)]{BrownGoodearlbook} and hence regular. The actual centre $Z$ is a complete intersection ring (and hence Gorenstein) by \cite[Theorem~21.3]{DeConProcesiQuantumGroups}. Moreover, it is shown in the proof of \cite[Theorem~21.5]{DeConProcesiQuantumGroups} that it is an integrally closed domain.  
\end{enumerate}
It is noted in \cite[Remark~5.4]{GordonNato} that the centre of (VI) Quantum Borels, considered in \cite{BGS}, is Gorenstein if and only if each simple factor of $G$ is of type $B_r,C_r,D_r$ ($r$ even) $E_6,E_7,E_8$ or $G_2$. When this is the case, the other hypothesis of Theorem~\ref{thm:rankonemainbody} hold (with $R = Z_+$ of \cite[Section 7]{BGS}) since the centre is actually smooth. However, the examples (VII) Quantized function algebras of \cite{BGS} do not satisfy the hypothesis of Theorem~\ref{thm:rankonemainbody} because their centres are not Gorenstein \cite[Remark~5.4]{GordonNato}.   


We note that example (V) is $\Z$-graded and the restricted quantum group $A$ is equal to $H / \mf{m} H$ for a graded maximal ideal $\mf{m} \lhd R$. Then $A$ admits a triangular decomposition, implying that it has the rank-one property with respect to baby Verma modules. This was also shown in \cite[Proposition~4.16(2)]{CoreThiel} using results from the literature on the multiplicities $[\Delta(\lambda) : L(\mu)]$. As explained in the proof of \cite[Proposition~4.16(2)]{CoreThiel}, this implies that Lusztig's small quantum group also satisfies the rank-one property. 

The results of \cite{BraunSymmetric} provide other important examples satisfying the hypothesis of Theorem~\ref{thm:rankonemainbody}. Firstly, we may take $H$ to be a non-commutative crepant resolution of an integrally closed Gorenstein domain over an algebraically closed field of characteristic zero; see \cite[Example~2.22]{BraunSymmetric}. Secondly, we may take $H$ to be a 3 or 4 dimensional PI Sklyanin algebra \cite[Example~2.24]{BraunSymmetric}. 

The main result of \cite{TransferFrobenius} says that, for a filtered algebra, being a free Frobenius extension lifts from the associated graded. This provides an effective way of checking the property for a large class of examples. 

Finally, we expect that quiver Hecke algebras (KLR algebras) \cite{KLRKhovanovLauda,RuquierKLR} also satisfy the hypothesis of Theorem~\ref{thm:rankonemain}. 

\subsection{Positive characteristic}\label{sec:poscharrankone}

As explained in \cite[Proposition~4.16(1)]{CoreThiel}, combining results in the literature on the representation theory of the restricted enveloping algebra of simple Lie algebras in positive characteristic allows one to show that these algebras have the rank-one property. Therefore, it is natural to ask to what extent our result extends to algebras in positive, or mixed, characteristic. In these situations, assumption \eqref{list:H5} of Section~\ref{sec:proofmainrankone} can fail. However, in the case $H = U(\mf{g})$, for $\mf{g}$ a simple Lie algebra over an algebraically closed field of characteristic $p > 0$, the second Kac-Weisfeiler conjecture \cite{PremetKW} provides an effective way of showing that \eqref{list:H5} holds for regular characters $\chi$. For rational Cherednik algebras at $t = 1$, the result \cite[Proposition~6.8]{MoPositiveChar} is similarly applicable. 

Rather, it is assumption \eqref{list:H6} that causes difficulties since the PI degree is always  zero in $\K$ in these examples. Recall that \eqref{list:H6} is used to argue that $Z$ and $Z_{\alpha}$ are summand of $H$. It raises the question: is $Z$ a direct summand of $H$ in either of these two examples? 

When $\mf{g} = \mf{psl}_n$, with $n = d p^m > 4$, $m \ge 1$ and $1 \le d < p$, it follows from \cite{BraunCentre} that $Z$ is not a direct summand of $H$. If $Z$ were a summand of $H$ then it would be Cohen-Macaulay since $H$ is Cohen-Macaulay, but that contradicts \cite[Theorem~L(1)]{BraunCentre}. We do not know whether $Z$ is a summand of $H$ when $\mf{g}$ also satisfies Jantzen's standard assumptions; see (1)-(3) in \cite[Introduction]{BraunCentre}.

\section{Acknowledgements}

The first author would like to thank Ken Brown and Lewis Topley for fruitful discussions. We would like to thank Amiram Braun for explaining that the answer to the final question of Section~\ref{sec:poscharrankone} is no. We would also like to thank the referee for a number of suggestions that improved the article. The first author was partially supported by a Research Project Grant from the Leverhulme Trust and by the EPSRC grant EP-W013053-1. This work is a contribution to the SFB-TRR 195 "Symbolic Tools in Mathematics and their Application'' of the German Research Foundation (DFG).

\def\cprime{$'$} \def\cprime{$'$} \def\cprime{$'$} \def\cprime{$'$}
\def\cprime{$'$} \def\cprime{$'$} \def\cprime{$'$} \def\cprime{$'$}
\def\cprime{$'$} \def\cprime{$'$} \def\cprime{$'$} \def\cprime{$'$}
\def\cprime{$'$} \def\cprime{$'$}


\begin{thebibliography}{10}
	
	\bibitem{MoPositiveChar}
	G.~Bellamy and M.~Martino.
	\newblock On the smoothness of centres of rational {C}herednik algebras in
	positive characteristic.
	\newblock {\em Glasg. Math. J.}, 55(A):27--54, 2013.
	
	\bibitem{CoreThiel}
	G.~Bellamy and U.~Thiel.
	\newblock Cores of graded algebras with triangular decomposition.
	\newblock {\em arXiv}, 1711.00780v1, 2017.
	
	\bibitem{BelThielHighest}
	G.~Bellamy and U.~Thiel.
	\newblock Highest weight theory for finite-dimensional graded algebras with
	triangular decomposition.
	\newblock {\em Adv. Math.}, 330:361--419, 2018.
	
	\bibitem{BonnafeRouquier}
	C.~Bonnaf\'e and R.~Rouquier.
	\newblock Cherednik algebras and {C}alogero-{M}oser cells.
	\newblock {\em arXiv}, 1708.09764v3, 2017.
	
	\bibitem{BraunSymmetric}
	A.~Braun.
	\newblock On symmetric, smooth and {C}alabi-{Y}au algebras.
	\newblock {\em J. Algebra}, 317(2):519--533, 2007.
	
	\bibitem{BraunCentre}
	A.~Braun.
	\newblock The center of the enveloping algebra of the {$p$}-{L}ie algebras
	{$\mathfrak{sl}_n$}, {$\mathfrak{pgl}_n$}, {$\mathfrak{psl}_n$}, when {$p$}
	divides {$n$}.
	\newblock {\em J. Algebra}, 504:217--290, 2018.
	
	\bibitem{BrownGoodearlbook}
	K.~A. Brown and K.~R. Goodearl.
	\newblock {\em {L}ectures on {A}lgebraic {Q}uantum {G}roups}.
	\newblock Advanced Courses in Mathematics. CRM Barcelona. Birkh\"auser Verlag,
	Basel, 2002.
	
	\bibitem{Ramifications}
	K.~A. Brown and I.~G. Gordon.
	\newblock The ramification of centres: {L}ie algebras in positive
	characteristic and quantised enveloping algebras.
	\newblock {\em Math. Z.}, 238(4):733--779, 2001.
	
	\bibitem{BGS}
	K.~A. Brown, I.~G. Gordon, and C.~H. Stroppel.
	\newblock Cherednik, {H}ecke and quantum algebras as free {F}robenius and
	{C}alabi-{Y}au extensions.
	\newblock {\em J. Algebra}, 319(3):1007--1034, 2008.
	
	\bibitem{CohMac}
	W.~Bruns and J.~Herzog.
	\newblock {\em Cohen-{M}acaulay rings}, volume~39 of {\em Cambridge Studies in
		Advanced Mathematics}.
	\newblock Cambridge University Press, Cambridge, 1993.
	
	\bibitem{DeConProcesiQuantumGroups}
	C.~De~Concini and C.~Procesi.
	\newblock Quantum groups.
	\newblock In {\em {$D$}-modules, representation theory, and quantum groups
		({V}enice, 1992)}, volume 1565 of {\em Lecture Notes in Math.}, pages
	31--140. Springer, Berlin, 1993.
	
	\bibitem{EG}
	P.~Etingof and V.~Ginzburg.
	\newblock Symplectic reflection algebras, {C}alogero-{M}oser space, and
	deformed {H}arish-{C}handra homomorphism.
	\newblock {\em Invent. Math.}, 147(2):243--348, 2002.
	
	\bibitem{GordonNato}
	I.~Gordon.
	\newblock Representations of semisimple {L}ie algebras in positive
	characteristic and quantum groups at roots of unity.
	\newblock In {\em Quantum groups and {L}ie theory ({D}urham, 1999)}, volume 290
	of {\em London Math. Soc. Lecture Note Ser.}, pages 149--167. Cambridge Univ.
	Press, Cambridge, 2001.
	
	\bibitem{KLRKhovanovLauda}
	M.~Khovanov and A.~D. Lauda.
	\newblock A diagrammatic approach to categorification of quantum groups {II}.
	\newblock {\em Trans. Amer. Math. Soc.}, 363(5):2685--2700, 2011.
	
	\bibitem{LamFirstcourse}
	T.~Y. Lam.
	\newblock {\em A first course in noncommutative rings}, volume 131 of {\em
		Graduate Texts in Mathematics}.
	\newblock Springer-Verlag, New York, second edition, 2001.
	
	\bibitem{TransferFrobenius}
	S.~Launois and L.~Topley.
	\newblock Transfer results for {F}robenius extensions.
	\newblock {\em J. Algebra}, 524:35--58, 2019.
	
	\bibitem{LorenzFDHopf}
	M.~Lorenz.
	\newblock Representations of finite-dimensional {H}opf algebras.
	\newblock {\em J. Algebra}, 188(2):476--505, 1997.
	
	\bibitem{LorenzBook}
	M.~Lorenz.
	\newblock {\em A tour of representation theory}, volume 193 of {\em Graduate
		Studies in Mathematics}.
	\newblock American Mathematical Society, Providence, RI, 2018.
	
	\bibitem{LorenzFitzgerald}
	M.~Lorenz and L.~Fitzgerald~Tokoly.
	\newblock Projective modules over {F}robenius algebras and {H}opf comodule
	algebras.
	\newblock {\em Comm. Algebra}, 39(12):4733--4750, 2011.
	
	\bibitem{MR}
	J.~C. McConnell and J.~C. Robson.
	\newblock {\em Noncommutative {N}oetherian {R}ings}, volume~30 of {\em Graduate
		Studies in Mathematics}.
	\newblock American Mathematical Society, Providence, RI, revised edition, 2001.
	\newblock With the cooperation of L. W. Small.
	
	\bibitem{PremetKW}
	A.~Premet.
	\newblock Irreducible representations of {L}ie algebras of reductive groups and
	the {K}ac-{W}eisfeiler conjecture.
	\newblock {\em Invent. Math.}, 121(1):79--117, 1995.
	
	
	\bibitem{RuquierKLR}
	R.~Rouquier.
	\newblock Quiver {H}ecke algebras and 2-{L}ie algebras.
	\newblock {\em Algebra Colloq.}, 19(2):359--410, 2012.
	
	\bibitem{CHAMP}
	U.~Thiel.
	\newblock Champ: a {C}herednik algebra {M}agma package.
	\newblock {\em LMS J. Comput. Math.}, 18(1):266--307, 2015.
	
\end{thebibliography}

\end{document}